\documentclass[reqno]{amsart}
\usepackage{amssymb,graphicx}%
\usepackage{ifpdf}
\ifpdf
 \usepackage[hyperindex]{hyperref}
\else
 \expandafter\ifx\csname dvipdfm\endcsname\relax
 \usepackage[hypertex,hyperindex]{hyperref}
 \else
 \usepackage[dvipdfm,hyperindex]{hyperref}

 \usepackage{lipsum}
\makeatletter
\g@addto@macro{\endabstract}{\@setabstract}
\newcommand{\authorfootnotes}{\renewcommand\thefootnote{\@fnsymbol\c@footnote}}%
\makeatother
 \fi
\fi
\allowdisplaybreaks[4]
\theoremstyle{plain}
\newtheorem{thm}{Theorem}[section]
\newtheorem{lem}{Lemma}[section]
\newtheorem{cor}{Corollary}[section]
\theoremstyle{remark}

\numberwithin{equation}{section}

\begin{document}

\title[L\MakeLowercase{ocal} A\MakeLowercase{ronson}-B\MakeLowercase{\'{e}nolan type gradient estimates }]
{\textbf{L\MakeLowercase{ocal} A\MakeLowercase{ronson}-B\MakeLowercase{\'{e}nolan type gradient estimates   for the porous medium type
equation under} R\MakeLowercase{icci flow}}}

\author[W. Wang, H. Zhou, D. Xie]{W\MakeLowercase{en} W\MakeLowercase{ang} \qquad H\MakeLowercase{ui} Z\MakeLowercase{hou} \qquad D\MakeLowercase{apeng} X\MakeLowercase{ie}}
\address[W. Wang$^{1,2}$, H. Zhou$^{1,2}$, D. Xie$^{1}$]{1. School of Mathematics and Statistics, Hefei Normal
University, Hefei 230601,P.R.China;}
\address{2. School of mathematical Science, University of Science and
Technology of China, Hefei 230026, China}
\email{\href{mailto: W. Wang <wwen2014@mail.ustc.edu.cn>}{wwen2014@mail.ustc.edu.cn}}

\today

\begin{abstract}
In this paper, we investigate some new local
Aronson-B\'{e}nilan
 type gradient estimates  for positive solutions of the
porous medium equation
$$
u_{t}=\Delta u^{m},
$$
under Ricci flow.
 As application, the related Harnack inequalities are derived.
 Our results generalize known results. These results in the paper can be regard as generalizing the gradient
estimates of Lu-Ni-V\'{a}zquez-Villani and Huang-Huang-Li to the Ricci flow.
\end{abstract}

\keywords{ porous medium equation; Gradient estimate; Harnack inequality}

\subjclass[2010]{ 58J35, 35K05,  53C21}

\thanks{This paper was typeset using \AmS-\LaTeX}

\thanks{Corresponding author: Wen Wang, E-mail: wwen2014@mail.ustc.edu.cn}

\maketitle
\section{\textbf{Introduction}}

In the paper, we mainly derive the parabolic version gradient estimates and Harnack inequality for positive solutions
to the porous medium equation (PME for short)
\begin{equation}\label{1.1}
u_{t}=\Delta u^{m}, \quad m>1
\end{equation}
under Ricci flow.

Let $(M^{n}, g)$ be a complete Riemannian manifold. Li and Yau \cite{9}  established a famous
gradient estimate for positive solutions to the heat equation.
In 1991,
Li in \cite{10} deduced gradient estimates and Harnack inequalities for positive solutions to the nonlinear parabolic equation on
$M\times [0,\infty)$.In $1993$, Hamilton in \cite{5}  generalized the constant $\alpha$ of Li and Yau's result to the function $\alpha(t)=e^{2Kt}$.
In $2006$, Sun \cite{17} also proved gradient estimates of different coefficient.
In $2011$, Li and Xu in \cite{11} further promoted Li and Yau's result, and found  two new functions $\alpha(t)$.
Recently, first author and Zhang in \cite{18} further generalized Li and Xu's results to the nonlinear parabolic equation.
Related results can be found in \cite{4,14,19}.

In 2009,  Lu, Ni, V\'{a}zquez and Villani in \cite{13} studied the PME
on manifolds,
and obtained the results below.

\textbf{Theorem A ( Lu, Ni, V\'{a}zquez and Villani)}. \emph{Let $(\mathbf{M}^{n}, g)$ be an $n$-dimensional complete Riemannnian
manifold with $\mathbf{Ric}(B_{p}(2R))\geq -K$, $K>0$. Assume that $u$ is a positive solution of ~\eqref{1.1}. Let $v=\frac{m}{m-1}u^{m-1}$
and $M=\max_{B_{p}(2R)\times [0,T]}v$. Then for any $\alpha>1$, we have
\begin{eqnarray*}
\nonumber
\frac{|\nabla v|^2}{v}-\alpha\frac{v_t}{v}&\leq  &\frac{CMa\alpha^2}{R^2}\left(\frac{\alpha^2}{\alpha-1}am^{2}+(m-1)(1+\sqrt{K}R)\right)\\
&&+\frac{\alpha^2}{\alpha-1}a(m-1)MK+\frac{a\alpha^2}{t}
\end{eqnarray*}
on the ball $B_{p}(2R)$, where $a=\frac{n(m-1)}{n(m-1)+2}$ and the constant $C$depends only on $n$.}

\emph{Moreover, when $R\rightarrow\infty$, the following gradient estimate on complete noncaompact
Riemannian manifold $(\mathbf{M}^{n}, g)$ can be deduced:
\begin{eqnarray*}
\frac{|\nabla v|^2}{v}-\alpha\frac{v_t}{v}&\leq  &\frac{\alpha^2}{\alpha-1}a(m-1)MK+\frac{a\alpha^2}{t}
\end{eqnarray*}}

 Huang,  Huang and  Li in \cite{7} generalized the results of Lu, Ni,  V\'{a}zquez and  Villani, obtained Li-Yau type, Hamilton type
  and Li-Xu type gradient estimates.

Recently, above these results had been generalized to the Ricci flow.

The Ricci flow
\begin{equation}\label{1.2}
\partial_{t}g(x,t)=-2Ric(x,t)
\end{equation}
was first introduced by Hamilton \cite{6} to investigate the Poincar\'{e} conjecture
on compact three dimensional manifolds of positive Ricci curvature.
In 2008, Kuang and Zhang \cite{8} proved a gradient estimate for positive solutions to the conjugate heat equation under
Ricci flow on a closed manifold. Soon afterwards,   gradient estimate for positive solutions to the heat equation
under Ricci flow were further studied, one can see \cite{1,12,14,16}.
Recently, Cao and Zhu \cite{3} derived some Aronson and B\'{e}nilan estimates for PME ~\eqref{1.1} under Ricci flow.

We first introduce three $C^1$ functions $\alpha(t)$, $\varphi(t)$ and $\gamma(t)$ $: (0,+\infty)\rightarrow (0,+\infty)$.
Suppose that three $C^1$ functions $\alpha(t)$, $\varphi(t)$ and $\gamma(t)$
  satisfy the following conditions:\\
  $(C1)$ $\alpha(t)>1$, $\varphi(t)$ and $\gamma(t)$.\\
  $(C2)$ $\alpha(t)$ and $\varphi(t)$ satisfy the following system
 \begin{equation*}
\left\{\aligned
\frac{2\varphi}{n(m-1)}-2(m-1)M K\geq (\frac{2\varphi}{n(m-1)}-\alpha')\frac{1}{\alpha},\\
\frac{2\varphi}{n(m-1)}-\alpha'>0,\\
\frac{\varphi^2}{n(m-1)}+\alpha\varphi'\geq 0.
\endaligned\right.
\end{equation*}
$(C3)$ $\gamma(t)$ satisfies
 \begin{equation*}
\frac{\gamma'}{\gamma}-(\frac{2\varphi}{n(m-1)}-\alpha')\frac{1}{\alpha}\leq 0.
\end{equation*}
 $(C4)$ $\alpha(t)$ and $\gamma(t)$ are non-decreasing.

\section{\textbf{Main Results}}
 Our results state as follows.

 \begin{thm}\quad Let $(M^{n}, g(x,t))_{t\in [0,T]}$ be a complete solution to the Ricci flow ~\eqref{1.2}. Let
 $M^n$ be complete under the initial metric $g(x,0)$.
  Assume that $|\mathrm{Ric}(x,t)|\leq K$ for some $K>0$ and all $t\in [0,T]$.
  Suppose that there exist three functions $\alpha(t)$, $\varphi(t)$ and $\gamma(t)$
  satisfy  conditions (C1), (C2), (C3) and (C4).

Given $x_{0}\in M$ and $R>0$, let $u$ be a positive solution of the nonlinear parabolic equation
~\eqref{1.1}
in the cube $B_{2R,T}:=\{(x,t)|d(x,x_{0},t)\leq 2R, 0\leq t\leq T\}$. Let $v=\frac{m}{m-1}u^{m-1}$ and $v\leq M$.

If $\frac{\gamma\alpha^4}{\alpha-1}\leq C_{2}$ for some constant $C_2$,
Then for any $(x,t)\in B_{2R,T}$,
\begin{eqnarray}\label{2.1}
\nonumber
\frac{|\nabla v|^2}{v}-\alpha\frac{v_t}{v}&\leq&  Ca\alpha^{2}\left[\frac{1}{R^2}\Big(1+\sqrt{K}R\Big)+K\right]+\frac{CaM m^{2}}{R^{2}\gamma}\\
&&+\alpha^{2}K\sqrt{a(m-1)}+\frac{K\alpha^{2}}{m-1}\sqrt{an}.
\end{eqnarray}

If $\frac{\gamma}{\alpha-1}\leq C_{2}$ for some constant $C_2$,
Then for any $(x,t)\in B_{2R,T}$,
\begin{eqnarray}\label{2.2}
\nonumber
\frac{|\nabla v|^2}{v}-\alpha\frac{v_t}{v}&\leq&  Ca\alpha^{2}\left[\frac{1}{R^2}\Big(1+\sqrt{K}R\Big)+K\right]+\frac{CaM m^{2}\alpha^{4}}{R^{2}\gamma}\\
&&+\alpha^{2}K\sqrt{a(m-1)}+\frac{K\alpha^{2}}{m-1}\sqrt{an}.
\end{eqnarray}
holds on $B_{2R,T}$, where $a=\frac{n(m-1)}{n(m-1)+1}$,
and the constant $C$ depends only on $n$.
\end{thm}

Let us show some special functions to illustrate the theorem $2.1$ holds for different circumstances and
see appendix in section $5$ for detailed calculation process.

\textbf{Remarks:}
1. Li-Yau type:
\begin{align*}
&\alpha(t)=constant,\quad \varphi(t)=\frac{\alpha n(m-1)}{t}+\frac{n(m-1)^{2}MK}{\alpha-1},\\
& \gamma(t)=t^{\theta}
\quad  with \quad 0<\theta\leq 2.
\end{align*}
Then
\begin{align*}
\frac{|\nabla v|^2}{v}-\alpha\frac{v_t}{v}&\leq  Ca\alpha^{2}\left[\frac{1}{R^2}\Big(1+\sqrt{K}R\Big)+K\right]+\frac{Ca m^{2}\alpha^{4}M}{R^{2}\gamma}\\
&+\alpha^{2}K\sqrt{a(m-1)}+\frac{K\alpha^{2}}{m-1}\sqrt{an}
\end{align*}

2. Hamilton type:
$$\alpha(t)=e^{2(m-1)MKt},\varphi(t)=\frac{n(m-1)}{t}e^{4(m-1)MKt},\gamma(t)=te^{2(m-1)MKt}.$$
Then
\begin{align*}
\frac{|\nabla v|^2}{v}-\alpha\frac{v_t}{v}&\leq  Ca\alpha^{2}\left[\frac{1}{R^2}\Big(1+\sqrt{K}R\Big)+K\right]+\frac{Ca m^{2}\alpha^{4}M}{R^{2}\gamma}\\
&+\alpha^{2}K\sqrt{a(m-1)}+\frac{K\alpha^{2}}{m-1}\sqrt{an}.
\end{align*}

3. Li-Xu type:
\begin{align*}
&\alpha(t)=1+\frac{\sinh((m-1)MKt)\cosh((m-1)MKt)-(m-1)MKt}{\sinh^{2}((m-1)MKt)},\\
&\varphi(t)=2n(m-1)^{2}MK[1+\coth((m-1)MKt)],\quad \gamma(t)=\tanh((m-1)MKt).
\end{align*}
Then
\begin{align*}
\frac{|\nabla v|^2}{v}-\alpha\frac{v_t}{v}&\leq  Ca\left[\frac{1}{R^2}\Big(1+\sqrt{K}R\Big)+K\right]+\frac{Ca m^{2}M}{R^{2}\gamma}\\
&+\alpha^{2}K\sqrt{a(m-1)}+\frac{K\alpha^{2}}{m-1}\sqrt{an},
\end{align*}
where $\alpha(t)$ is bounded uniformly.

4. Linear Li-Xu type:
\begin{align*}
&\alpha(t)=1+(m-1)MKt,\quad \varphi(t)=\frac{n(m-1)}{t}+n(m-1)^{2}MK,\\
&\gamma(t)=(m-1)MKt.
\end{align*}
Then
\begin{align*}
\frac{|\nabla v|^2}{v}-\alpha\frac{v_t}{v}&\leq  Ca\alpha^{2}\left[\frac{1}{R^2}\Big(1+\sqrt{K}R\Big)+K\right]+\frac{Ca m^{2}\alpha^{4}M}{R^{2}\gamma}\\
&+\alpha^{2}K\sqrt{a(m-1)}+\frac{K\alpha^{2}}{m-1}\sqrt{an}.
\end{align*}

The local estimates above imply global estimates.
\begin{cor}
Let $(M^{n}, g(0))$ be a complete noncompact Riemannian manifold without
boundary, and asuume $g(t)$ evolves by Ricci flow in such a way that $|\mathrm{Ric}|\leq K$ for $t\in[0,T]$.
Let $u(x,t)$ be a positive solution to the equation ~\eqref{1.1}, and let $v=\frac{m}{m-1}u^{m-1}$ and $v\leq M$.

If $l\leq 1$ and for $(x,t)\in M^{n}\times (0,T]$, then
\begin{equation*}
\frac{|\nabla v|^2}{v}-\alpha\frac{v_t}{v}\leq Ca\alpha^{2}K
+\alpha^{2}K\sqrt{a(m-1)}+\frac{K\alpha^{2}}{m-1}\sqrt{an}.
\end{equation*}
\end{cor}

\section{\textbf{Preliminary}}
Let $v=\frac{m}{m-1}u^{m-1}$ and put into equation (1.1), we get
\begin{equation}\label{3.1}
v_{t}=(m-1)v\Delta v +|\nabla v|^{2},
\end{equation}
which is equivalent to the following form:
\begin{equation}\label{3.2}
\frac{v_{t}}{v}=(m-1)\Delta v +\frac{|\nabla v|^{2}}{v}.
\end{equation}

\begin{lem}
Assume that $(M^{n}, g(x,t))$ satisfies the hypotheses of Theorem $2.1$. We introduce the differential operator
\begin{equation*}
\mathcal{L}=\partial_{t}-(m-1)v\Delta .
\end{equation*}
Let $F=\frac{|\nabla v|^{2}}{v}-\alpha \frac{v_t}{v}-\alpha\varphi$, where $\alpha=\alpha(t)>1$. Then we have
\begin{eqnarray}\label{3.3}
\nonumber
\mathcal{L}(F)&\leq&-(m-1)\sum_{i,j}^{n} v^{2}_{ij}+(m-1)\alpha^{2}K^{2}+2(m-1)K|\nabla v|^{2}+2m\nabla v\nabla F\\
&&-[(m-1)\Delta v]^{2}+2(\alpha-1)K\frac{|\nabla v|^2}{v}-\alpha'\frac{v_t}{v}-\alpha'\varphi-\alpha\varphi'.
\end{eqnarray}
\end{lem}

\begin{proof}\quad Simple calculation shows
\begin{eqnarray}\label{3.4}
\nonumber
\partial_{t}\left(\frac{v_t}{v}\right)&=& \partial_{t}\left[(m-1)\Delta v+\frac{|\nabla v|^{2}}{v}\right]\\
\nonumber
&=&(m-1)(\Delta v)_{t}+\frac{2\nabla v \nabla v_{t}}{v}+\frac{2}{v}\mathrm{Ric}(\nabla v,\nabla v)-\frac{|\nabla v|^{2}v_t}{v^2}\\
\nonumber
&=&(m-1)(\Delta v)_{t}+\frac{2\nabla v}{v}\nabla\Big[(m-1)v\Delta v +|\nabla v|^{2}\Big]\\
\nonumber
&&-\frac{|\nabla v|^{2}}{v}\Big[(m-1)v\Delta v +|\nabla v|^{2}\Big]+\frac{2}{v}\mathrm{Ric}(\nabla v,\nabla v)\\
\nonumber
&=&(m-1)(\Delta v)_{t}+(m-1)\frac{|\nabla v|^{2}\Delta v}{v}+2(m-1)\nabla v\nabla(\Delta v)\\
&&+\frac{2\nabla v\nabla|\nabla v|^{2}}{v}
+\frac{(m-1)\lambda|\nabla v|^{2}}{v}-\frac{|\nabla v|^4}{v^2}+\frac{2}{v}\mathrm{Ric}(\nabla v,\nabla v),
\end{eqnarray}
and
\begin{eqnarray}\label{3.5}
\Delta\left(\frac{v_t}{v}\right)&=& \frac{\Delta(v_t)+2\sum_{i,j}R_{ij}f_{ij}}{v}-\frac{2\nabla v\nabla v_t}{v^2}-\frac{v_{t}\Delta v}{v^2}+\frac{2|\nabla v|^{2}v_{t}}{v^3},
\end{eqnarray}
where we use the fact that
\begin{eqnarray}\label{3.6}
(\Delta v)_{t}=\Delta(v_{t})+2\sum_{i,j}^{n}R_{ij}v_{ij},
\end{eqnarray}
\begin{eqnarray}\label{3.7}
(|\nabla v|^2)_{t}=2\nabla v\nabla(v_t)+2Ric(\nabla v, v).
\end{eqnarray}
Combining \eqref{3.4} and \eqref{3.5}, we have
\begin{eqnarray}\label{3.8}
\nonumber
\mathcal{L}\left(\frac{v_t}{v}\right)&=& (m-1)\frac{|\nabla v|^{2}\Delta v}{v}+2(m-1)\nabla v\nabla(\Delta v)+\frac{2\nabla v\nabla|\nabla v|^2}{v}\\
\nonumber
&& -\frac{|\nabla v|^4}{v^2}
+2(m-1)\frac{\nabla v\nabla v_t}{v}+(m-1)\frac{v_{t}\Delta v}{v}-2(m-1)\frac{|\nabla v|^{2}v_t}{v^2}\\
&&+2(m-1)\sum_{i,j}R_{ij}v_{ij}+\frac{2}{v}\mathrm{Ric}(\nabla v,\nabla v).
\end{eqnarray}
Since
\begin{align*}
\nabla v_{t}=&(m-1)\nabla v\Delta v+(m-1)v\nabla(\Delta v)+\nabla|\nabla v|^{2},
\end{align*}
then
\begin{align*}
&2(m-1)\nabla v\nabla(\Delta v)+\frac{2\nabla v\nabla|\nabla v|^2}{v}+(m-1)\frac{|\nabla v|^{2}\Delta v}{v}-\frac{|\nabla v|^4}{v^2}\\
&=\frac{2}{v}\nabla v \nabla v_{t} -\frac{|\nabla v|^{2}}{v}\frac{v_t}{v}.
\end{align*}
Substituting above identity and the following identity
\begin{equation*}
2(m-1)\frac{\nabla v\nabla v_t}{v}-2(m-1)\frac{|\nabla v|^{2}v_{t}}{v^2}~=~2(m-1)v\nabla\left(\frac{v_t}{v}\right)\nabla\log v
\end{equation*}
into \eqref{3.8}, we have
\begin{eqnarray}\label{3.9}
\nonumber
\mathcal{L}\left(\frac{v_t}{v}\right)&=&\frac{2}{v}\nabla v \nabla v_{t} -\frac{|\nabla v|^{2}}{v}\frac{v_t}{v}
+2(m-1)v\nabla\left(\frac{v_t}{v}\right)\nabla\log v\\
&& +(m-1)\frac{v_{t}\Delta v}{v}++2(m-1)\sum_{i,j}R_{ij}v_{ij}+\frac{2}{v}\mathrm{Ric}(\nabla v,\nabla v).
\end{eqnarray}
On the other hand, similar calculations show
\begin{eqnarray}\label{3.10}
\nonumber
\partial_{t}\left(\frac{|\nabla v|^2}{v}\right)&=& \frac{2\nabla v}{v}\nabla\Big[(m-1)v\Delta v+|\nabla v|^{2}\Big]
+\frac{2Ric(\nabla v,\nabla v)}{v}\\
\nonumber
&&-\frac{|\nabla v|^2}{v}\left[(m-1)\Delta v+\frac{|\nabla v|^2}{v}+\right]\\
\nonumber
&=&2(m-1)\Delta v\frac{|\nabla v|^2}{v}+2(m-1)\nabla v\nabla(\Delta v)+\frac{2}{v}\nabla v\nabla|\nabla v|^{2}\\
&&-(m-1)\Delta v\frac{|\nabla v|^2}{v}-\frac{|\nabla v|^4}{v^2}+\frac{2Ric(\nabla v,\nabla v)}{v},
\end{eqnarray}
where we utilize the formula \eqref{3.7} in \eqref{3.10}. \\
By utilize Bochner's formula, we have
\begin{eqnarray}\label{3.11}
\nonumber
\Delta\left(\frac{|\nabla v|^2}{v}\right)&=& \frac{2v^{2}_{ij}}{v}+\frac{2\nabla v\Delta(\nabla v)}{v}-
2\frac{\nabla v\nabla|\nabla v|^2}{v^2}-\Delta v\frac{|\nabla v|^{2}}{v^2}+\frac{2|\nabla v|^4}{v^3}\\
\nonumber
&=&\frac{2v^{2}_{ij}}{v}+2R_{ij}\frac{|\nabla v|^{2}}{v}+\frac{2}{v}\nabla v\nabla(\Delta v)
-\Delta v\frac{|\nabla v|^2}{v^2}\\
&&-2\nabla\left(\frac{|\nabla v|^2}{v}\right)\nabla(\log v).
\end{eqnarray}
From  \eqref{3.10} and \eqref{3.11}, we obtain
\begin{eqnarray}\label{3.12}
\nonumber
\mathcal{L}\left(\frac{|\nabla v|^2}{v}\right)&=&2(m-1)\Delta v\frac{|\nabla v|^2}{v}+\frac{2}{v}\nabla v\nabla|\nabla v|^{2}
 -2(m-1)R_{ij}|\nabla v|^{2}\\
 \nonumber
 &&-2(m-1)v^{2}_{ij}
 -\frac{|\nabla v|^4}{v^2}+2(m-1)v\nabla\left(\frac{|\nabla v|^2}{v}\right)\nabla(\log v)\\
 &&+\frac{2Ric(\nabla v,\nabla v)}{v}.
\end{eqnarray}
By utilize \eqref{3.9} and \eqref{3.12}, we have
\begin{eqnarray}\label{3.13}
\nonumber
&&\mathcal{L}(F)=\mathcal{L}\left(\frac{|\nabla v|^2}{v}\right)
-\alpha\mathcal{L}\left(\frac{v_t}{v}\right)-\alpha'\frac{v_t}{v}-\alpha'\varphi-\alpha\varphi'\\
\nonumber
&=&
2(m-1)\Delta v\frac{|\nabla v|^2}{v}+\frac{2}{v}\nabla v\nabla|\nabla v|^{2}
 -2(m-1)(v^{2}_{ij}+\alpha R_{ij}v_{ij})\\
\nonumber
&&-2(m-1)R_{ij}|\nabla v|^{2}
-\frac{|\nabla v|^4}{v^2}+2(m-1)v\nabla\left(\frac{|\nabla v|^2}{v}\right)\nabla(\log v)\\
\nonumber
&&
-\alpha\frac{2}{v}\nabla v\nabla v_{t}+\alpha\frac{|\nabla v|^2}{v}\frac{v_t}{v}
-2\alpha(m-1)v\nabla\left(\frac{v_t}{v}\right)\nabla(\log v)\\
&&
-\alpha(m-1)\Delta v\frac{v_t}{v}-\frac{2(\alpha-1)}{v}\mathrm{Ric}(\nabla v,\nabla v)
-\alpha'\varphi-\alpha\varphi'.
\end{eqnarray}
It is not difficult to calculate that
\begin{eqnarray}\label{3.14}
\nonumber
&&2(m-1)v\nabla\left(\frac{|\nabla v|^2}{v}\right)\nabla(\log v)-2\alpha(m-1)\nabla\left(\frac{v_t}{v}\right)\nabla(\log v)\\
\nonumber
&=&2(m-1)\nabla v \nabla\left[\frac{|\nabla v|^2}{v}-\alpha\frac{v_t}{v}\right]\\
&=&2(m-1)\nabla v\nabla F,
\end{eqnarray}
and
\begin{eqnarray}\label{3.15}
\nonumber
\frac{2}{v}\nabla v\nabla|\nabla v|^{2}-\alpha \frac{2}{v}\nabla v\nabla v_{t}
&=&\frac{2}{v}\nabla v\nabla(|\nabla v|^{2}-\alpha  v_{t})\\
\nonumber
&=&\frac{2}{v}\nabla v\nabla (Fv)\\
&=&2\nabla v\nabla F+2F\frac{|\nabla v|^2}{v}.
\end{eqnarray}
We deduce from \eqref{3.14} and \eqref{3.15} that
\begin{eqnarray}\label{3.16}
\nonumber
&&2(m-1)v\nabla\left(\frac{|\nabla v|^2}{v}\right)\nabla(\log v)-2\alpha(m-1)\nabla\left(\frac{v_t}{v}\right)\nabla(\log v)\\
\nonumber
&&+\frac{2}{v}\nabla v\nabla|\nabla v|^{2}-\alpha \frac{2}{v}\nabla v\nabla v_{t}\\
\nonumber
&=&2m\nabla v\nabla F+2F\frac{|\nabla v|^2}{v}\\
&=&2m\nabla v\nabla F+2\left(\frac{|\nabla v|^2}{v}-\alpha\frac{v_t}{v}\right)\frac{|\nabla v|^2}{v},
\end{eqnarray}
and
\begin{eqnarray}\label{3.17}
\nonumber
&&2(m-1)\Delta v\frac{|\nabla v|^2}{v}-\frac{|\nabla v|^4}{v^2}-\alpha(m-1)\Delta v\frac{v_t}{v}
+\alpha\frac{v_t}{v}\frac{|\nabla v|^2}{v}\\
\nonumber
&=&2\frac{|\nabla v|^2}{v}\left(\frac{v_t}{v}-\frac{|\nabla v|^2}{v}\right)-\frac{|\nabla v|^4}{v^2}-\alpha\frac{v_t}{v}\left(\frac{v_t}{v}-\frac{|\nabla v|^2}{v}\right)+\alpha\frac{v_t}{v}\frac{|\nabla v|^2}{v}\\
&=&(2\alpha+2)\frac{v_t}{v}\frac{|\nabla v|^2}{v}-3\frac{|\nabla v|^4}{v^2}-\alpha\left(\frac{v_t}{v}\right)^{2}.
\end{eqnarray}
From \eqref{3.16} and \eqref{3.17}, we have
\begin{eqnarray}\label{3.18}
\nonumber
&&2(m-1)v\nabla\left(\frac{|\nabla v|^2}{v}\right)\nabla(\log v)-2\alpha(m-1)\nabla\left(\frac{v_t}{v}\right)\nabla(\log v)
+\frac{2}{v}\nabla v\nabla|\nabla v|^{2}\\
\nonumber
 &&-\alpha \frac{2}{v}\nabla v\nabla v_{t}+2(m-1)\Delta v\frac{|\nabla v|^2}{v}-\frac{|\nabla v|^4}{v^2}-\alpha(m-1)\Delta v\frac{v_t}{v}+\alpha\frac{v_t}{v}\frac{|\nabla v|^2}{v}\\
 \nonumber
&=&2m\nabla v\nabla F-\left(\frac{v_t}{v}-\frac{|\nabla v|^2}{v}\right)^{2}+(1-\alpha)\left(\frac{v_t}{v}\right)^{2}\\
&\leq&2m\nabla v\nabla F-[(m-1)\Delta v]^{2}\quad for \quad \alpha>1.
\end{eqnarray}
Substituting  \eqref{3.18} into \eqref{3.13}, we arrive at
\begin{align}\label{3.19}
\nonumber
\mathcal{L}(F)\leq&-2(m-1)(v^{2}_{ij}+\alpha R_{ij}v_{ij})-2(m-1)R_{ij}|\nabla v|^{2}\\
&+2m\nabla v\nabla F-[(m-1)\Delta v]^{2}
-2(\alpha-1)R_{ij}\frac{|\nabla v|^2}{v}-\alpha'\varphi-\alpha\varphi'.
\end{align}
Further, applying Young's inequality
\begin{eqnarray*}
|R_{ij}||v_{ij}|\leq \frac{\alpha}{2}R^{2}_{ij}+\frac{1}{2\alpha}v^{2}_{ij}
\end{eqnarray*}
 to (3.19), we conclude
We complete the proof of Lemma $3.1$.
\end{proof}

\begin{lem} Suppose that $(M^{n}, g(t))_{t\in[0, T]}$ satisfies the hypotheses of Theorem $2.1$.
We also assume that $\alpha(t)>1$ and $\varphi(t)>0$ satisfy the following system
\begin{equation}\label{3.20}
\left\{\aligned
\frac{2\varphi}{n(m-1)}-2(m-1)M K\geq (\frac{2\varphi}{n(m-1)}-\alpha')\frac{1}{\alpha},\\
\frac{2\varphi}{n(m-1)}-\alpha'>0,\\
\frac{\varphi^2}{n(m-1)}+\alpha\varphi'\geq 0,
\endaligned\right.
\end{equation}
and $\alpha(t)$ is non-decreasing.
Then
\begin{align}\label{3.21}
\nonumber
\mathcal{L} F\leq& -(m-1)\sum_{i,j}^{n} \left[v_{ij}+\frac{\varphi}{n(m-1)}\delta_{ij}\right]^{2}
-\left[\frac{2\varphi}{n(m-1)}-\alpha'\right]\frac{1}{\alpha}F\\
&+(m-1)\alpha^{2}K^{2}+2(\alpha-1)K\frac{|\nabla v|^{2}}{v}+2m\nabla v\nabla F
-[(m-1)\Delta v]^{2}.
\end{align}
\end{lem}

\begin{proof}
By utilizing the unit matrix $(\delta_{ij})_{n\times n}$ and $(3.3)$, we obtain
\begin{eqnarray*}
\nonumber
\mathcal{L}(F)&\leq&-(m-1)\sum_{i,j}^{n} \Big[v^{2}_{ij}+\frac{\varphi}{n(m-1)}\delta^{2}_{ij}\Big]
+\frac{\varphi^2}{n(m-1)}+\frac{2\varphi}{n}\Delta v\\
&&+(m-1)\alpha^{2}K^{2}+2(m-1)KM\frac{|\nabla v|^{2}}{v}+2m\nabla v\nabla F\\
&&-[(m-1)\Delta v]^{2}+2(\alpha-1)K\frac{|\nabla v|^2}{v}-\alpha'\frac{v_t}{v}-\alpha'\varphi-\alpha\varphi'.
\end{eqnarray*}
Applying (3.2) to above inequality, we have
\begin{eqnarray*}
\nonumber
\mathcal{L}(F)&\leq&-(m-1)\sum_{i,j}^{n} \Big[v^{2}_{ij}+\frac{\varphi}{n(m-1)}\delta^{2}_{ij}\Big]
-\Big[\frac{2\varphi}{n(m-1)}-2(m-1)M K\Big]\frac{|\nabla v|^{2}}{v}\\
&&+\Big[\frac{2\varphi}{n(m-1)}-\alpha'\Big]\frac{v_t}{v}
+\Big[\frac{2\varphi}{n(m-1)}-\alpha'\Big]\frac{\alpha \varphi}{\alpha}
+\frac{\varphi^2}{n(m-1)}\\
&&-\Big[\frac{2\varphi}{n(m-1)}-\alpha'\Big]\frac{\alpha \varphi}{\alpha}+(m-1)\alpha^{2}K^{2}+2(m-1)KM\frac{|\nabla v|^{2}}{v}\\
&&+2m\nabla v\nabla F-[(m-1)\Delta v]^{2}+2(\alpha-1)K\frac{|\nabla v|^2}{v}-\alpha'\varphi-\alpha\varphi'.
\end{eqnarray*}
Again using (3.20), we follows (3.21).
\end{proof}

\begin{lem}
Let $G=\gamma(t)F$. Then
\begin{eqnarray}\label{3.22}
\nonumber
\mathcal{L} G&\leq &-\frac{1}{a\alpha^{2}\gamma}G^{2}
+\left[\frac{\gamma'}{\gamma}-\left(\frac{2\varphi}{n(m-1)}-\alpha'\right)\frac{1}{\alpha}\right]G\\
\nonumber
&&-\frac{2(\alpha-1)}{n\alpha^{2}}\frac{|\nabla v|^2}{v}G-\frac{\gamma(m-1)(\alpha-1)^2}{n\alpha^2}\frac{|\nabla v|^4}{v^2}
\\
&&+(m-1)\alpha^{2}\gamma K^{2}+2\gamma(\alpha-1)K\frac{|\nabla v|^{2}}{v}
+2m\nabla v\nabla G,
\end{eqnarray}
where $a=\frac{n(m-1)}{n(m-1)+1}$.
\end{lem}

\begin{proof} Sample calculation gives
\begin{eqnarray}\label{3.23}
\nonumber
\mathcal{L} G&=&\gamma\mathcal{L}F+\gamma' F\\
\nonumber
&\leq &-(m-1)\gamma\left[v^{2}_{ij}+\frac{\varphi}{n(m-1)}\delta^{2}_{ij}\right]
+\left[-\left(\frac{2\varphi}{n(m-1)}-\alpha'\right)\frac{1}{\alpha}+\frac{\gamma'}{\gamma}\right]G\\
\nonumber
&&+(m-1)\alpha^{2}\gamma K^{2}+2\gamma(\alpha-1)K\frac{|\nabla v|^{2}}{v}
+2m\nabla v\nabla G\\
&&-\gamma[(m-1)\Delta v]^{2}.
\end{eqnarray}
Since
\begin{align}\label{3.24}
\nonumber
\left[v^{2}_{ij}+\frac{\varphi}{n(m-1)}\delta^{2}_{ij}\right]&\geq \frac{1}{n}(\Delta v+\varphi)\\
&=\frac{1}{n\alpha^{2}(m-1)^2}\left[F+(m-1)(\alpha-1)\frac{|\nabla v|^2}{v}\right]^{2},
\end{align}
and
\begin{eqnarray}\label{3.25}
\nonumber
(m-1)\Delta v&=&-\frac{F}{\alpha}-\frac{\alpha-1}{\alpha}\frac{|\nabla v|^{2}}{v}-\varphi\\
&\leq &-\frac{F}{\alpha}.
\end{eqnarray}
Therefore, we follow that from (3.23), (3.24) and (3.25)
\begin{eqnarray}\label{3.26}
\nonumber
\mathcal{L} G&\leq &-\frac{\gamma}{n\alpha^{2}(m-1)}\left[F+(m-1)(\alpha-1)\frac{|\nabla v|^2}{v}\right]^{2}\\
\nonumber
&&
+\left[-\left(\frac{2\varphi}{n(m-1)}-\alpha'\right)\frac{1}{\alpha}+\frac{\gamma'}{\gamma}\right]G+(m-1)\alpha^{2}\gamma K^{2}\\
&&+2\gamma(\alpha-1)K\frac{|\nabla v|^{2}}{v}
+2m\nabla v\nabla G
-\frac{G^2}{\alpha^{2}\gamma}.
\end{eqnarray}
From (3.26), we infer (3.22). The proof is complete.
\end{proof}

\section{\textbf{Proof of Main Results}}
In this section, we will prove our main results.

\begin{proof}[\textbf{Proof of Theorem $2.1$}]\quad
Now let $\varphi(r)$ be a $C^2$ function on $[0,\infty)$ such that
\begin{equation*}
\varphi(r)=\left\{\aligned
1 \quad if ~r\in[0,1],\\
0 \quad if ~r\in [2,\infty),
\endaligned\right.
\end{equation*}
and
$$0\leq \varphi(r)\leq 1, \quad \varphi'(r)\leq 0,\quad \varphi''(r)\leq 0, \quad
\frac{|\varphi'(r)|}{\varphi(r)}\leq C,$$
where $C$ is an absolute constant.
Let define by
$$\phi(x,t)=\varphi(d(x,x_{0},t))=\varphi\left(\frac{d(x,x_{0},t)}{R}\right)
=\varphi\left(\frac{\rho(x,t)}{R}\right),$$
where $\rho(x,t)=d(x,x_{0},t)$.
By using the maximum principle, the argument of Calabi \cite{2} allows us to
suppose that the function $\phi(x,t)$ with support in $B_{2R,T}$, is $C^2$ at the maximum point.
By utilize the Laplacian theorem, we deduce that
\begin{eqnarray}\label{4.1}
\frac{|\nabla \phi|^{2}}{\phi}&\leq& \frac{C}{R^2},
\end{eqnarray}
\begin{eqnarray}\label{4.2}
-\Delta \phi&\leq& \frac{C}{R^2}(1+\sqrt{K}R),
\end{eqnarray}

For any $0\leq T_{1}\leq T$, let $H=\phi G$ and $(x_{1},t_{1})$ be the point in $B_{2R,T_1}$
at which $G$ attain its maximum value. We can suppose that
the value is positive, because otherwise the proof is trivial.
Then at the point $(x_{1}, t_{1})$, we infer
\begin{eqnarray}\label{4.3}
\mathcal{L}(H)\geq 0,\qquad \nabla G=-\frac{G}{\phi}\nabla\phi.
\end{eqnarray}
By the evolution formula of the geodesic length under the Ricci flow \cite{4}, we calculate
\begin{align*}
\phi_{t}G=&-G\phi'\left(\frac{\rho}{R}\right)\frac{1}{R}\frac{d\rho}{dt}
=G\phi'\left(\frac{\rho}{R}\right)\int_{\gamma_{t_1}}\mathrm{Ric}(S,S)ds\\
\leq &G\phi'\left(\frac{\rho}{R}\right)\frac{1}{R}K_{2}\rho\leq G\phi'\left(\frac{\rho}{R}\right)K_{2}\leq G\sqrt{C}K_{2},
\end{align*}
where $\gamma_{t_1}$ is the geodesic connecting $x$ and $x_0$ under the metric
$g(t_1)$, $S$ is the unite tangent vector to $\gamma_{t_1}$, and $ds$ is the
element of the arc length. Hence, by applying \eqref{4.2},  we have
\begin{eqnarray}\label{4.4}
\nonumber
0&\leq &\mathcal{L}(H)\leq~\phi\mathcal{L}G-(m-1)G\left(\Delta\phi-2\frac{|\nabla\phi|^2}{\phi}\right)+\phi_{t}G\\
\nonumber
&\leq &-\frac{1}{a\alpha^{2}\gamma}\phi G^{2}
+\left[\frac{\gamma'}{\gamma}-\left(\frac{2\varphi}{n(m-1)}-\alpha'\right)\frac{1}{\alpha}\right]\phi G\\
\nonumber
&&-\frac{2(\alpha-1)}{n\alpha^{2}}\frac{|\nabla v|^2}{v}\phi G-\frac{\gamma(m-1)(\alpha-1)^2}{n\alpha^2}\frac{|\nabla v|^4}{v^2}\phi
\\
\nonumber
&&+(m-1)\alpha^{2}\gamma\phi K^{2}++2\gamma\phi(\alpha-1)K\frac{|\nabla v|^{2}}{v}
+2m\phi\nabla v\nabla G\\
&&-(m-1)G\left(\Delta\phi-2\frac{|\nabla\phi|^2}{\phi}\right)+\sqrt{C}KG
\end{eqnarray}
Multiply $\phi$, we have
\begin{eqnarray*}
\nonumber
0&\leq &-\frac{1}{a\alpha^{2}\gamma}\phi^{2} G^{2}
+\left[\frac{\gamma'}{\gamma}\phi-\left(\frac{2\varphi}{n(m-1)}-\alpha'\right)\frac{\phi}{\alpha}\right]\phi G\\
\nonumber
&&-\frac{2(\alpha-1)}{n\alpha^{2}}\frac{|\nabla v|^2}{v}\phi^{2} G-\frac{\gamma(m-1)(\alpha-1)^2}{n\alpha^2}\frac{|\nabla v|^4}{v^2}\phi^{2}
\\
\nonumber
&&+(m-1)\alpha^{2}\gamma\phi^{2} K^{2}+2\gamma\phi^{2}(\alpha-1)K\frac{|\nabla v|^{2}}{v}\\
&&-2m\phi^{2}\frac{\nabla\phi}{\phi}G\nabla v-(m-1)\phi G\left(\Delta\phi-2\frac{|\nabla\phi|^2}{\phi}\right)+\sqrt{C}K\phi G
\end{eqnarray*}
Further using the inequality $Ax^{2}+Bx\geq -\frac{B^2}{4A}$ with $A>0$, we have
\begin{equation*}
-\frac{2(\alpha-1)}{n\alpha^{2}}\frac{|\nabla v|^2}{v}\phi^{2} G-2m\phi^{2}\frac{\nabla\phi}{\phi}G\nabla v
\leq \frac{nm^{2}\alpha^{2}}{2(\alpha-1)}\frac{|\nabla\phi|^2}{\phi}\phi G,
\end{equation*}
\begin{equation*}
-\frac{\gamma(m-1)(\alpha-1)^2}{n\alpha^2}\frac{|\nabla v|^4}{v^2}\phi^{2}++2\gamma\phi^{2}(\alpha-1)K\frac{|\nabla v|^{2}}{v}
\leq\frac{n\alpha^{2}K^{2}}{m-1}\phi^{2}\gamma.
\end{equation*}
Hence, we deduce that
\begin{eqnarray}\label{4.5}
\nonumber
0&\leq &-\frac{1}{a\alpha^{2}\gamma}\phi^{2} G^{2}
+\left[\frac{\gamma'}{\gamma}\phi-\left(\frac{2\varphi}{n(m-1)}-\alpha'\right)\frac{\phi}{\alpha}
+ \frac{nm^{2}\alpha^{2}}{2(\alpha-1)}\frac{|\nabla\phi|^2}{\phi}\right.\\
\nonumber
&&\left.
+(m-1)\left(\Delta\phi-2\frac{|\nabla\phi|^2}{\phi}\right)+\sqrt{C}K\right]\phi G\\
&&+(m-1)\alpha^{2}K^{2}\phi^{2}\gamma+\frac{n\alpha^{2}K^{2}}{m-1}\phi^{2}\gamma.
\end{eqnarray}
Combine (4.1), (4.2) and (4.5), we have
\begin{eqnarray*}
0&\leq &-\frac{1}{a\alpha^{2}\gamma}\phi^{2} G^{2}
+\Big[\frac{\gamma'}{\gamma}\phi-\big(\frac{2\varphi}{n(m-1)}-\alpha'\big)\frac{\phi}{\alpha}
+ \frac{Cm^{2}\alpha^{2}}{R^{2}(\alpha-1)}\\
&&
+\frac{C(m-1)}{R^2}(1+\sqrt{K}R)+\sqrt{C}K\Big]\phi G\\
&&+(m-1)\alpha^{2}K^{2}\phi^{2}\gamma+\frac{n\alpha^{2}K^{2}}{m-1}\phi^{2}\gamma.
\end{eqnarray*}
This inequality becomes
\begin{align*}
&\frac{1}{a\alpha^{2}\gamma}\phi^{2} G^{2}-\Big[\frac{\gamma'}{\gamma}\phi-\big(\frac{2\varphi}{n(m-1)}-\alpha'\big)\frac{\phi}{\alpha}
+ \frac{Cm^{2}\alpha^{2}}{R^{2}(\alpha-1)}\\
&
+\frac{C(m-1)}{R^2}(1+\sqrt{K}R)+\sqrt{C}K\Big]\phi G\\
&\leq
(m-1)\alpha^{2}K^{2}\phi^{2}\gamma+\frac{n\alpha^{2}K^{2}}{m-1}\phi^{2}\gamma.
\end{align*}
For the inequality $Ax^{2}-2Bx\leq C$, one has $x\leq \frac{2B}{A}+\left(\frac{C}{A}\right)^{\frac{1}{2}}$, where $A, B, C>0$.
\begin{eqnarray*}
\nonumber
\phi G(x,T_1)&\leq& (\phi G)(x_{1},t_1)\\
\nonumber
&\leq& \Big\{Ca\alpha^{2}\gamma\Big[\frac{m^{2}\alpha^{2}}{R^{2}(\alpha-1)}+\frac{(m-1)}{R^2}(1+\sqrt{K}R)+K\Big]\\
&&+a\alpha^{2}\gamma\phi\Big[\frac{\gamma'}{\gamma}-\big(\frac{2\varphi}{n(m-1)}-\alpha'\big)\frac{1}{\alpha}\Big]\\
&&+\alpha^{2}K\gamma\phi\sqrt{a(m-1)}+\frac{K\alpha^{2}\gamma}{m-1}\phi\sqrt{an}\Big\}(x_{1},t_1).
\end{eqnarray*}
If $\gamma$ is nondecreasing which satisfies the system
\begin{equation*}
\left\{\aligned
\frac{\gamma'}{\gamma}-(\frac{2\varphi}{n}-\alpha')\frac{1}{\alpha}\leq 0,\\
\frac{\gamma\alpha^4}{\alpha-1}\leq C.
\endaligned\right.
\end{equation*}
Recall that $\alpha(t)$ and $\gamma(t)$ are non-decreasing and $t_{1}<T_{1}$. Hence, we have
\begin{eqnarray*}
\nonumber
\phi G(x,T_1)&\leq& (\phi G)(x_{1},t_1)\\
\nonumber
&\leq& Ca\alpha^{2}(T_1)\gamma(T_1)\left[\frac{1}{R^2}\Big(1+\sqrt{K}R\Big)+K\right]+\frac{Ca m^{2}}{R^2}\\
&&+\alpha^{2}(T_1)K\gamma(T_1)\phi\sqrt{a(m-1)}+\frac{K\alpha^{2}(T_1)\gamma(T_1)}{m-1}\phi\sqrt{an}.
\end{eqnarray*}
Hence, we have for  $\phi\equiv 1$ on $B_{R,T}$,
\begin{eqnarray*}
F(x,T_1)&\leq&Ca\alpha^{2}(T_1)\left[\frac{1}{R^2}\Big(1+\sqrt{K}R\Big)+K\right]+\frac{Ca m^{2}}{R^{2}\gamma(T_1)}\\
&&+\alpha^{2}(T_1)K\sqrt{a(m-1)}+\frac{K\alpha^{2}(T_1)}{m-1}\sqrt{an}.
\end{eqnarray*}

If $\gamma$ is nondecreasing which satisfies the system
\begin{equation*}
\left\{\aligned
\frac{\gamma'}{\gamma}-(\frac{2\varphi}{n}-\alpha')\frac{1}{\alpha}\leq 0,\\
\frac{\gamma}{\alpha-1}\leq C.
\endaligned\right.
\end{equation*}
Recall that $\alpha(t)$ and $\gamma(t)$ are non-decreasing and $t_{1}<T_1$. Hence, we have
\begin{eqnarray*}
\nonumber
\phi G(x,T_1)&\leq& (\phi G)(x_{1},t_1)\\
\nonumber
&\leq& Ca\alpha^{2}(T_1)\gamma(T_1)\left[\frac{1}{R^2}\Big(1+\sqrt{K}R\Big)+K\right]+\frac{Ca m^{2}\alpha^{4}(T_1)}{R^2}\\
&&+\alpha^{2}(T_1)K\gamma(T_1)\phi\sqrt{a(m-1)}+\frac{K\alpha^{2}(T_1)\gamma(T_1)}{m-1}\phi\sqrt{an}.
\end{eqnarray*}
Hence, we have for  $\phi\equiv 1$ on $B_{R,T}$,
\begin{eqnarray*}
F(x,T_1)&\leq&Ca\alpha^{2}(T_1)\left[\frac{1}{R^2}\Big(1+\sqrt{K}R\Big)+K\right]+\frac{Ca m^{2}\alpha^{4}(T_1)}{R^{2}\gamma(T_1)}\\
&&+\alpha^{2}(T_1)K\sqrt{a(m-1)}+\frac{K\alpha^{2}(T_1)}{m-1}\sqrt{an}.
\end{eqnarray*}
Because $T_1$ is arbitrary, so the conclusion is valid.
 \end{proof}

\section{\textbf{Appendix}}
We will check some functions $\alpha(t)>1$, $\varphi(t)>0$ and $\gamma(t)>0$ in Remark  satisfy the following two systems
\begin{equation}\label{5.1}
\left\{\aligned
\frac{2\varphi}{n(m-1)}-2(m-1)MK\geq (\frac{2\varphi}{n(m-1)}-\alpha')\frac{1}{\alpha},\\
\frac{2\varphi}{n(m-1)}-\alpha'>0,\\
\frac{\varphi^2}{n(m-1)}+\alpha\varphi'\geq 0.
\endaligned\right.
\end{equation}
and
\begin{equation}\label{5.2}
\left\{\aligned
\frac{\gamma'}{\gamma}-(\frac{2\varphi}{n(m-1)}-\alpha')\frac{1}{\alpha}\leq 0,\\
\frac{\gamma\alpha^4}{\alpha-1}\leq C, ~or~ \frac{\gamma}{\alpha-1}\leq C.
\endaligned\right.
\end{equation}
Besides, $\alpha(t)$ and $\gamma(t)$ are non-decreasing.

$(1)$ Let $\alpha(t)=1+(m-1)MKt$,  $\varphi(t)=\frac{n(m-1)}{t}+n(m-1)^{2}MK$
and $\gamma(t)=(m-1)MKt$.

Direct calculation shows
\begin{align*}
(\mathrm{i})\quad &\frac{2\varphi}{n(m-1)}-\alpha'\\
&=\frac{2}{t}+2(m-1)MK-(m-1)MK>0,\\
(\mathrm{ii})\quad &\frac{\varphi^2}{n(m-1)}+\alpha\varphi'
=\frac{n(m-1)}{t^2}+n(m-1)^{3}M^{2}K^{2}+\frac{2}{t}n(m-1)^{2}MK\\
&
+\Big[1+(m-1)MKt\Big]\Big[-\frac{n(m-1)}{t^2}\Big]>0
\\
(\mathrm{iii})\quad &\alpha\Big(\frac{2\varphi}{n(m-1)}-2(m-1)M K\Big)-\Big(\frac{2\varphi}{n(m-1)}-\alpha'\Big)\\
&=(\alpha-1)\frac{2\varphi}{n(m-1)}-2(m-1)MK\alpha+\alpha'\\
&=2(m-1)MK+2(m-1)^{2}M^{2}K^{2}t-2(m-1)MK\\
&-2(m-1)^{2}M^{2}K^{2}t+\alpha'>0..
\end{align*}
Hence, $\alpha(t)=1+\frac{1}{3}(m-1)MKt$,   $\varphi(t)=\frac{n(m-1)}{t}+\frac{1}{3}n(m-1)^{2}MK$
 satisfy the system (5.1).

On the other hand, one has
\begin{eqnarray*}
&&\frac{\gamma'}{\gamma}-(\frac{2\varphi}{n(m-1)}-\alpha')\frac{1}{\alpha}\\
&=&\frac{1}{t}-\left[\frac{2}{t}+\frac{2}{3}(m-1)MK-\frac{1}{3}(m-1)MK\right]\frac{1}{1+\frac{1}{3}(m-1)MKt}\\
&=&-\frac{1}{t(1+\frac{2}{3}(m-1)MKt)}\\
&\leq & 0,\quad for \quad t\geq 0.
\end{eqnarray*}
and $\frac{\gamma}{\alpha-1}=1$. So, (5.2) is also satisfied.

$(2)$ $\alpha(t)=e^{2(m-1)MKt}$,  $\varphi(t)=\frac{n(m-1)}{t}e^{4(m-1)MKt}$ and $\gamma(t)=te^{2(m-1)MKt}$.
Direct calculation shows
\begin{align*}
(\mathrm{i})\quad &\frac{2\varphi}{n(m-1)}-\alpha'=\frac{2}{t}e^{2(m-1)MKt}(e^{2(m-1)MKt}-(m-1)MKt)>0,\\
(\mathrm{ii})\quad &\frac{\varphi^2}{n(m-1)}+\alpha\varphi'=\frac{n(m-1)}{t^2}e^{6(m-1)MKt}(e^{2(m-1)MKt}-1+4(m-1)MKt)>0,\\
(\mathrm{iii})\quad &\frac{2\varphi}{n(m-1)}-2(m-1)M K-(\frac{2\varphi}{n(m-1)}-\alpha')\frac{1}{\alpha}\\
&=\frac{2}{t}e^{4(m-1)MKt}-2(m-1)MK-\frac{2}{t}e^{2(m-1)MKt}+2(m-1)MK\\
&=\frac{2}{t}e^{2(m-1)MKt}(e^{2(m-1)MKt}-1)\geq 0.
\end{align*}
Hence, $\alpha(t)=e^{2(m-1)MKt}$ and  $\varphi(t)=\frac{n(m-1)}{t}e^{4(m-1)MKt}$ satisfy the system (5.1).

Besides, we have
\begin{eqnarray*}
&&\frac{\gamma'}{\gamma}-(\frac{2\varphi}{n(m-1)}-\alpha')\frac{1}{\alpha}\\
&=&\frac{1+2(m-1)MKt}{t}-\left(\frac{2}{t}e^{2(m-1)MKt}-2(m-1)MK\right)\\
&=&\frac{1}{t}(1+4(m-1)MKt-2e^{2(m-1)MKt})\\
&\leq & 0,\quad for \quad t\geq 0.
\end{eqnarray*}
and as $t\rightarrow 0^{+}$, $\frac{\gamma}{\alpha-1}=\frac{te^{2Kt}}{e^{2Kt}-1}\rightarrow\frac{1}{2K}$. This
implies $\frac{\gamma}{\alpha-1}\leq C$. So, (5.2) is also satisfied.

$(3)$ $\alpha(t)=1+\frac{\sinh((m-1)MKt)\cosh((m-1)MKt)-(m-1)MKt}{\sinh^{2}((m-1)MKt)}$,
$\varphi(t)=2n(m-1)^{2}MK[1+\coth((m-1)MKt)]$ and $\gamma(t)=\tanh((m-1)MKt)$.
Direct calculation shows
$$\alpha'(t)=2(m-1)MK-2(\alpha-1)(m-1)MK\coth[(m-1)MKt].$$
Then
\begin{align*}
(\mathrm{i})\quad &\frac{2\varphi}{n(m-1)}-\alpha'=2(m-1)MK[1+\coth((m-1)MKt)]-2(m-1)MK\\
&\quad\quad \quad\quad+2(\alpha-1)(m-1)MK\coth[(m-1)MKt]>0,\\
(\mathrm{ii})\quad&\alpha(\frac{2\varphi}{n(m-1)}-2(m-1)M K)-(\frac{2\varphi}{n(m-1)}-\alpha')\\
=&2(m-1)MK\alpha[1+\coth((m-1)MKt)]-2(m-1)MK\alpha\\
&-2(m-1)MK[1+\coth(m-1)MKt]+\alpha'\\
=&2(m-1)MK(\alpha-1)[1+\coth((m-1)MKt)]\\
&-2(m-1)MK\alpha+\alpha'\\
=&2(m-1)MK(\alpha-1)\coth((m-1)MKt)-2(m-1)MK+\alpha'=0\\
(\mathrm{iii})\quad &\frac{\varphi^2}{n(m-1)}+\alpha\varphi'\\
=&\frac{n(m-1)^{3}M^{2}K^{2}}{\sinh^{2}(m-1)MKt}\Big([1+\coth(m-1)MKt]^{2}\sinh^{2}(m-1)MKt-\alpha\Big).
\end{align*}
Let $x=(m-1)MKt$, then
\begin{align*}
&[1+\coth(m-1)MKt]^{2}\sinh^{2}(m-1)MKt-\alpha\\
&=\frac{e^{2x}}{(e^{2x}-1)^2}\Big[e^{4x}-2e^{2x}+3+4x\Big].
\end{align*}
Let $f(x)=e^{4x}-2e^{2x}+3+4x$ with $x\leq 0$. Obviously, $f(0)>0$ and
\begin{eqnarray*}
f'(x)&=&2(e^{4x}-e^{2x}+2)>0.
\end{eqnarray*}
Then we get $f(x)>0$ for $x>0$. Hence, we have
\begin{eqnarray*}
\frac{\varphi^2}{n(m-1)}+\alpha\varphi'>0.
\end{eqnarray*}
Hence, $\alpha(t)=1+\frac{\sinh((m-1)MKt)\cosh((m-1)MKt)-(m-1)MKt}{\sinh^{2}((m-1)MKt)}$ and
$\varphi(t)=2n(m-1)^{2}MK[1+\coth((m-1)MKt)]$ satisfy the system (5.1).

On the other hand, as $t\rightarrow 0$, we have $\frac{\gamma\alpha^4}{\alpha-1}\rightarrow 2$;
 $\frac{\gamma\alpha^4}{\alpha-1}\rightarrow 1$ for $t\rightarrow \infty$. These imply $\frac{\gamma\alpha^4}{\alpha-1}\leq C$, here $C$ is a universal constant.\\
 Besides, we have
\begin{eqnarray*}
&&\frac{\gamma'}{\gamma}-(\frac{2\varphi}{n(m-1)}-\alpha')\frac{1}{\alpha}\\
&=&\frac{1}{\alpha}\left[\frac{x\alpha}{\sinh(xt)\cosh(xt)}-2x-2x(1+\alpha)\coth(xt)\right]\\
&=&\frac{1}{\alpha}\left[\frac{x}{\sinh(xt)\cosh(xt)}[\alpha-2(1+\alpha)\cosh^{2}(xt)]-2x\right]\\
&=&\frac{1}{\alpha}\left[\frac{x}{\sinh(xt)}[\alpha(1-2\cosh(xt))-2\cosh(xt)]-2K\right]\\
&\leq & 0,\quad for \quad t\geq 0,
\end{eqnarray*}
where $x=(m-1)MK$.
So, (5.2) is also satisfied.

$(4)$ $\alpha(t)=constant$,  $\varphi(t)=\frac{\alpha n(m-1)}{t}+\frac{n(m-1)^{2}MK}{\alpha-1}$ and $\gamma(t)=t^{\theta}$ with $0<\theta\leq 2$.
Direct calculation gives
\begin{align*}
(\mathrm{i})\quad &\frac{2\varphi}{n(m-1)}-\alpha'=\frac{2\alpha}{t}+\frac{(m-1)MK}{\alpha-1}>0,\\
(\mathrm{ii})\quad &\frac{\varphi^2}{n(m-1)}+\alpha\varphi'=\frac{n(m-1)\alpha^2}{t^2}-\frac{n(m-1)\alpha^2}{t^2}\\
&+\frac{n^{2}(m-1)^{4}M^{2}K^{2}}{(\alpha-1)^2}+\frac{2\alpha n^{2}(m-1)^{3}MK}{(\alpha-1)t}>0,\\
(\mathrm{iii})\quad &\alpha\Big(\frac{2\varphi}{n(m-1)}-2(m-1)M K\Big)-(\frac{2\varphi}{n(m-1)}-\alpha')\\
&=(\alpha-1)\frac{2\varphi}{n(m-1)}-2(m-1)MK>0.
\end{align*}
Hence, $\alpha(t)=constant$, and $\varphi(t)=\frac{\alpha n(m-1)}{t}+\frac{n(m-1)^{2}MK}{\alpha-1}$  satisfy the system (5.1).

 On the other hand, we have
\begin{eqnarray*}
\frac{\gamma'}{\gamma}-(\frac{2\varphi}{n(m-1)}-\alpha')\frac{1}{\alpha}
&=&\frac{\theta}{t}-\frac{2}{t}-\frac{(m-1)MK}{(\alpha-1)\alpha}\\
&\leq & 0,\quad for \quad t\geq 0 \quad and \quad 0<\theta\leq 2.
\end{eqnarray*}
So, (5.2) is also satisfied.

\section{\textbf{Acknowledgement}}
We are grateful to Professor Jiayu Li for his encouragement. We also thanks Professor Qi S Zhang for introduction of this
problem in the summer course.


\end{document}